\documentclass[11pt, notitlepage]{article}
\usepackage{amssymb,amsmath,comment}
\catcode`\@=11 \@addtoreset{equation}{section}
\def\thesection{\arabic{section}}

\def\theequation{\thesection.\arabic{equation}}
\catcode`\@=12
\usepackage{colortbl}%

\usepackage{a4wide}
\def\R{\mathbb{R}}

\newcommand{\ds} {\displaystyle}
\newcommand{\e}{\varepsilon}
\newcommand{\pa} {\partial}
\newcommand{\al} {\alpha}

\newcommand{\de} {\delta}
\newcommand{\ga} {\gamma}
\newcommand{\Ga} {\Gamma}
\newcommand{\Om} {\Omega}
\newcommand{\ra} {\rightarrow}

\newcommand{\De} {\Delta}
\newcommand{\la} {\lambda}
\newcommand{\La} {\Lambda}
\newcommand{\noi} {\noindent}

\newcommand{\mb} {\mathbb}
\newcommand{\mc} {\mathcal}
\newcommand{\lra} {\longrightarrow}
\newcommand{\ld} {\langle}
\newcommand{\rd} {\rangle}

\usepackage[all]{xy}
\catcode`\@=11
\def\theequation{\@arabic{\c@section}.\@arabic{\c@equation}}
\catcode`\@=12

\def\R{{I\!\!R}}

\def\QED{\hfill {$\square$}\goodbreak \medskip}

\newtheorem{Theorem}{Theorem}[section]
\newtheorem{Lemma}[Theorem]{Lemma}
\newtheorem{Proposition}[Theorem]{Proposition}
\newtheorem{Corollary}[Theorem]{Corollary}

\newtheorem{Definition}[Theorem]{Definition}

\begin{document}
\vspace{0.01in}

\title
{On the first curve of  Fu\v{c}ik Spectrum Of $p$-fractional Laplacian Operator with nonlocal normal boundary conditions}

\author{Divya Goel\footnote{Department of Mathematics, Indian Institute of Technology Delhi, Hauz Khas, New Delhi-110016, India.  e-mail: divyagoel2511@gmail.com},  Sarika Goyal\footnote{Department of Mathematics, Bennett University, Greater Noida, Uttar Pradesh - 201310, India.\hspace{15mm} e-mail: sarika1.iitd@gmail.com}  and  K. Sreenadh\footnote{Department of Mathematics, Indian Institute of Technology Delhi,
		Hauz Khaz, New Delhi-110016, India.
		e-mail: sreenadh@maths.iitd.ac.in} }

\date{}

\maketitle

\begin{abstract}

In this article, we study the Fu\v{c}ik spectrum of the $p$-fractional
Laplace operator with nonlocal normal derivative conditions which is defined as the set of all $(a,b)\in
\mb
R^2$ such that
$$
\mc (F_p)\left\{
\begin{array}{lr}
\Lambda_{n,p}(1-\al)(-\De)_{p}^{\al} u  + |u|^{p-2}u = \frac{\chi_{\Om_\e}}{\e} (a (u^{+})^{p-1} - b (u^{-})^{p-1}) \;\quad  \text{in}\;
\Om,\quad \\
\mc{N}_{\al,p} u = 0 \; \quad  \mbox{in}\; \mb R^n \setminus \overline{\Om},
\end{array}
\right.
$$
has a non-trivial solution $u$, where $\Om$ is a bounded domain in
$\mb R^n$ with Lipschitz boundary, $p \geq 2$, $n>p \al $, $\e, \al \in(0,1)$ and $\Om{_\e}:=\{x \in \Om:  d(x,\pa \Om)\leq \e \}$. We showed existence of the first non-trivial curve $\mc C$ of this spectrum which is used to obtain the  variational characterization of a second eigenvalue of the problem $\mc (F_p)$.  We also discuss some  properties of this curve $\mc C$, e.g. Lipschitz continuous,
strictly decreasing and asymptotic behaviour and nonresonance with respect to the Fu\v{c}ik spectrum.

\medskip

\noi \textbf{Key words:} nonlocal operator, Fu\v{c}ik spectrum, Steklov problem, Non-resonance.

\medskip

\noi \textit{2010 Mathematics Subject Classification:} 35A15, 35J92, 35J60.

\end{abstract}

\bigskip
\vfill\eject

\section{Introduction}
The Fu\v{c}ik spectrum of $p$-fractional Laplacian with nonlocal normal derivative is defined as
the set $\sum_p$ of all $(a,b)\in \mb
 R^2$ such that
{\small \begin{equation}\label{eq1}
 \La_{n,p}(1-\al)(-\De)_{p}^{\al} u  + |u|^{p-2}u = \frac{\chi_{\Om_\e}}{\e} (a (u^{+})^{p-1} - b (u^{-})^{p-1}) \;\text{in}\;
\Om,\; \mc{N}_{\al,p} u = 0 \;  \mbox{in}\; \mb R^n \setminus \overline{\Om},
\end{equation}}
has a non-trivial solution $u$, where $\Om$ is a bounded domain in $\mb R^n$ with Lipschitz boundary, $p \geq 2$, $\al, \e \in (0,1)$ and $\Om{_\e}:=\{x \in \Om:  d(x,\pa \Om)\leq \e \}$. The $(-\De)^{\al}_p$
is the $p$-fractional Laplacian operator defined as
\begin{equation*}
(-\De)_{p}^{\al} u(x):= 2\; \text{p.v} \int_{\mb
R^n}\frac{|u(x)-u(y)|^{p-2}(u(x)-u(y))}{|x-y|^{n+p\al}} dy \; \text{for all} \;
x\in \mb R^n,
\end{equation*}
and $\mc{N}_{\al,p}$ is the associated nonlocal derivative defined in \cite{va} as
\begin{equation*}
\mc{N}_{\al,p}u(x):= 2  \int_{\Omega} \frac{|u(x)-u(y)|^{p-2}(u(x)-u(y))}{|x-y|^{n+p\al}} dy \quad \text{for all} \;
x\in \mb R^n \setminus \overline{\Om} .
\end{equation*}
\noi In \cite{bre}, Bourgain, Brezis and Mironescu
 proved that for any smooth bounded domain  $\Om \subset \mb R^n,$ $ u \in W^{1,p}(\Om)$, there exist a constant $\La_{n,p}$ such that
\begin{equation*}
\lim_{\al \ra 1 ^-} \La_{n,p} (1-\al)\int _{\Om \times \Om }\frac{|u(x)-u(y)|^{p}}{|x-y|^{n+p\al}} dx  dy= \int_{\Om}|\nabla u|^p dx.
\end{equation*}
The constant $\Lambda_{n,p}$ can be explicitly computed and is given by
\begin{equation*}
\La_{n,p}= \frac{p\Gamma(\frac{n+p}{2})}{2\pi^{\frac{n-1}{2}}\Gamma(\frac{p+1}{2}) }.
\end{equation*}
For $a=b=\la $, the Fu\v{c}ik spectrum in \eqref{eq1} becomes the usual spectrum that satisfies
\begin{equation}\label{eq2}
\La_{n,p}(1-\al)(-\De)_{p}^{\al} u  + |u|^{p-2}u = \frac{ \la}{\e} \chi_{\Om_\e} |u|^{p-2}u  \; \text{in}\;
\Om,\quad\; \mc{N}_{\al,p} u = 0 \;  \mbox{in}\; \mb R^n \setminus \overline{\Om}.
\end{equation}
\noi In \cite{pe}, authors proved that there exists a sequence of eigenvalues $\la_{n,\e}(\Om_{\e})$ of \eqref{eq2}  such that $\la_{n,\e}(\Om_{\e})\ra \infty$ as $n\ra \infty.$ Moreover, $0<\la_{1,\e}(\Om_{\e})<\la_{2,\e}(\Om_{\e})\leq...\leq\la_{n,\e}(\Om_{\e})\leq...$  and the first eigenvalue $\la_{1,\e}(\Om_{\e})$ of \eqref{eq2} is simple, isolated and can be characterized as follows
\[\la_{1,\e}(\Om_{\e}) = \inf_{u\in \mc W^{\al,p}}\left\{  \La_{n,p}(1-\al)\int_{Q}\frac{|u(x)-u(y)|^p}{|x-y|^{n+p\al}}dxdy + \int_{\Om} |u|^p dx :  \int_{\Om_{\e}}|u|^p dx= \e \right\}.\]
 
\noi The  Fu\v{c}ik spectrum was introduced by Fu\v{c}ik (1976) who studied the problem in one dimension with periodic boundary
conditions.  In higher dimensions, the non-trivial first curve in the Fu\v{c}ik spectrum of Laplacian with Dirichlet boundary for bounded domain has been studied in \cite{fg}. Later
in \cite{cfg} Cuesta, de Figueiredo and Gossez studied this problem for $p$-Laplacian operator with Dirichlet boundary condition. The Fu\v{c}ik spectrum in the case of Laplacian, $p$-Laplacian operator with Dirichlet, Neumann and Robin boundary condition has been studied by many authors, for instance \cite{al, cg, ro, ros, pe,kpe}. In \cite{sa}, Goyal and Sreenadh extended the results of \cite{cfg} to nonlocal linear operators which include fractional Laplacian. The existence of  Fu\v{c}ik eigenvalues for $p$-fractional Laplacian operator with Dirichlet boundary conditions has been studied by many authors, for instance refer \cite{dan, inf}. Also, in \cite{hardy}, Goyal discussed the Fu\v{c}ik spectrum of of $p$-fractional Hardy Sobolev-Operator with weight function. A non-resonance
problem with respect to Fu\v{c}cik spectrum is also discussed in many papers \cite{cfg,kpr,mP}.
\\

\noi The inspiring point of our work is \cite{sa, hardy}, where the existence of a nontrivial curve is studied only for $p=2$ but the nature of the curve is left open for $p\ne 2$. In the present work, we extend the results obtained in \cite{sa} to the nonlinear case of $p$-fractional operator for any $p \geq 2$ and also show that this curve is the first curve.  We also showed the variational characterization of the second eigenvalue of the operator associated  with \eqref{eq1}. There is a substantial difference while handling the nonlinear nature of the operator. This difference is reflected while constructing the paths below a mountain-pass level (see the proof of Theorem 1.1).  To the best of our knowledge, no work has been done on the Fu\v{c}ik spectrum for nonlocal operators with nonlocal normal derivative. We would like to remark that the the main result obtained in this paper is new even for the following $p$-fractional Laplacian equation with Dirichlet boundary condition:
\[(-\De)_{p}^{\alpha} u + |u|^{p-2} u = a (u^+)^{p-1}-b (u^-)^{p-1}\;\; \text{in}\;\; \Om, \;\; u=0 \;\; \text{on}\; \mathbb{R}^n\backslash \overline{\Om}.\]
\noi With this introduction, we state our main result:
\begin{Theorem}\label{f31}
Let $s\geq 0$ then the point $(s+c(s),c(s))$ is the first nontrivial
point of $\sum_{p}$ in the intersection between $\sum_{p}$ and the line
$(s,0)+t(1,1)$ of $\mc F_p$.
\end{Theorem}
\noi The paper is organized as follows: In section 2 we give some notations and preliminaries. In section 3 we construct a
first nontrivial curve in $\sum_{p}$, described as $(s+c(s),c(s))$.
In section 4 we prove that the lines $\la_{1,\e}(\Om_{\e})\times \mb R$ and $\mb
R\times\la_{1,\e}(\Om_{\e})$ are isolated in $\sum_{p}$, the curve that we obtained
in section 3 is the first nontrivial curve and give the variational
characterization of second eigenvalue of \eqref{eq1}.  In section 5 we
prove some properties of the first curve and non resonance problem.

\section{Preliminaries}
\noi In this section we assemble some requisite material.
\noindent By \cite{va} we know the nonlocal analogue of divergence theorem which states that for any bounded functions $u$ and $v\in C^2$, it holds that
\begin{equation*}
\int_{\Om}(-\De)_{p}^{\al} u(x) dx = -\int _{\Om ^c}\mc{N}_{\al,p}u(x) dx.
\end{equation*}
More generally, we have following integration by parts  formula
\begin{equation*}
\mc H_{\al,p}(u,v)=\int_{\Om}v(x)(-\De)_{p}^{\al} u(x) dx+\int _{\Om ^c} v(x) \mc{N}_{\al,p}u(x) dx,
\end{equation*}
where $\mathcal{H}_{\al,p}(u,v)$ is defined as
\begin{equation*}
\mc H_{\al,p}(u,v):= \int _Q \frac{|u(x)-u(y)|^{p-2}(u(x)-u(y))(v(x)-v(y))}{|x-y|^{n+p\al}} dy, \qquad Q:= \mb R^{2n} \setminus (\Om^c)^2.
\end{equation*}
Now, given a measurable function $ u: \mb R^n\ra \mb R $, we set
\begin{equation} \label{feq12}
\|u\|_{\al,p}:= (\|u\|^p_{L^p(\Om)}+[u]^p_{\al,p})^{\frac{1}{p}},\quad  \text{where } [u]_{\al,p}:=(\mc H_{\al,p}(u,u))^{\frac{1}{p}}.
\end{equation}
\noi Then $\|.\|_{\al,p}$ defines a norm on the space
\begin{equation*}
\mc W^{\al,p}:= \{ u: \mb R^n\ra \mb R \text{ measurable }:\|u\|_{\al,p} < \infty \}.
\end{equation*}
Clearly $\mc W^{\al,p}\subset W^{\al,p}(\Om)$,  where $ W^{\al,p}(\Om)$ denotes the usual fractional Sobolev space endowed with the norm \[\|u\|_{W^{\al,p}}=\|u\|_{L^p}+ \left(\int_{\Om\times\Om} \frac{(u(x)-u(y))^{p}}{|x-y|^{n+p \al }}dxdy \right)^{\frac 1p}.\]
To study the fractional Sobolev space in detail see \cite{ms}.

\begin{Definition}
A function $u \in \mc {W}^{\al,p}$ is a weak solution of \eqref{eq1}, if for every $v\in \mc {W}^{\al,p}$, \\$u$
satisfies
\begin{align*}
 \La_{n,p}(1-\al) \mathcal{H}_{\al,p}(u,v) 
 + \int_{\Om}|u|^{p-2}uv-\frac{a}{\e}
\int_{\Om_{\e}} (u^{+})^{ p-1}v+\frac{b}{\e}\int_{\Om_{\e}} (u^{-})^{ p-1}v=0.
\end{align*}
\end{Definition}
\noi Now, we define the functional $J$ associated to the problem \eqref{eq1} as $J : \mc {W}^{\al,p} \ra \mb R$ such that 
\[  J(u)=  \La_{n,p}(1-\al) \int_{Q}\frac{|u(x)-u(y)|^p}{|x-y|^{n+p\al}} dxdy + \int_{\Om}|u|^p dx - \frac{a}{\e} \int_{\Om_{\e}} (u^{+})^p dx+\frac{b}{\e} \int_{\Om_{\e}} (u^{-})^p dx.\] 
Then $J$ is Fr$\acute{e}$chet differentiable in $ \mc {W}^{\al,p}$ and {for all $v\in \mathcal{W}^{\al,p}$.}
\[
 \langle J^\prime (u),v\rangle=\La_{n,p}(1-\al) \mathcal{H}_{\al,p}(u,v) + \int_{\Om}|u|^{p-2}uv-\frac{a}{\e}
\int_{\Om_{\e}} (u^{+})^{ p-1}v+\frac{b}{\e}
\int_{\Om_{\e}} (u^{-})^{ p-1}v.\]
\noindent For the sake of completeness, we describe the Steklov problem
\begin{equation}\label{stek}
(-\De)_{p} u  + |u|^{p-2}u = 0 \;\quad  \text{in}\;
\Om,\quad
 |\nabla u|^{p-2}\frac{\partial u}{\partial \nu}=\la |u|^{p-2}u  \; \quad  \mbox{on}\; \partial \Om,
\end{equation}
\noi where $\Om$ is a bounded domain and $p>1$. By \cite{pe}, \eqref{eq1} is related to \eqref{stek} in the sense that if $\Om $ be a bounded smooth domain in $\mb R^n$ with Lipschitz boundary and $p\in(1,\infty)$. For a fixed
$u\in W^{1,p}(\Om)\setminus W^{1,p}_0(\Om)$, we have 
\begin{equation*}
\begin{aligned}
\lim_{\e \ra 0^+}\frac{1}{\e}\int_{\Om_{\e}}|u|^pdx=\int_{\pa \Om}|u|^pdS \quad \text{and}& \lim_{\al \ra 1^-}\La_{n,p}(1-\al)[Eu]^p_{\al,p}= \|\nabla u\|^p_{L^p(\Om)},
\end{aligned}
\end{equation*}
\noi  where $E$ is a bounded linear extension operator from $W^{1,p}(\Om)$ to $W^{1,p}_0(B_R)$ such that $Eu=u$ in $\Om$ and $\Om$ is relatively compact in $B_R$, the ball of radius $R$ in $\mathbb{R}^n$. This leads to the following Lemma of \cite{ros}.
\begin{Lemma}\label{f4}
 Let $\Om $ be a smooth domain in $\mb \R^n$ with Lipschitz boundary and $p \in (1,\infty)$. For a fixed
$u\in W^{1,p}(\Om)\setminus W^{1,p}_0(\Om)$, it holds
\begin{equation*}
\lim_{\al \ra 1^-}\frac{\La_{n,p}(1-\al)[Eu]^p_{\al,p}+\|Eu\|^p_{L^p(\Om)}}{\frac{1}{1-\al}\|Eu\|^p_{L^p(\Om_{1-\al})}}= \frac{\|\nabla u\|^p_{L^p(\Om)}+\|u\|^p_{L^p(\Om)}}{\|u\|^p_{L^p(\partial \Om)}}.
\end{equation*}
\end{Lemma}

\noi Taking $\e=1-\al$, by Lemma \ref{f4} the eigenvalue $\la_{1,1-\al}(\Om_{1-\al}) \ra \la_1$ as $\al \ra 1^-$, where $\la_1$ is the first eigenvalue of the operator associated with \eqref{stek}. Similarly, we obtain  that as $ \al \ra 1^-$ the Fu\v{c}ik Spectrum of the operator associated with \eqref{eq1} tends to Fu\v{c}ik Spectrum of the Steklov problem.\\

\noi We shall throughout use the function space $\mc {W}^{\al,p}$ with the norm
$\|.\|$ and we use the standard $L^{p}(\Om)$ space whose norms
are denoted by $\|u\|_{L^p(\Om)}$. Also, we denote $\la_{n,\e}(\Om_{\e})$ by $\la_{n,\e}$. and  $\phi_{1,\e}$ is the eigenfunction  corresponding to $\la_{1,\e}$.

\section{The Fu\v{c}ik spectrum $\sum_{p}$}
In this section, we study  existence of the first nontrivial curve in the Fu\v{c}ik spectrum $\sum_{p}$ of \eqref{eq1}. We find that the points
in $\sum_{p}$ are associated with the critical value of some restricted functional. For  this, for fixed $s\in \mb R$ and  $s\geq 0$, we consider the
functional $J_{s}: \mc{W}^{\al,p} \ra \mb R $ defined by
\begin{align*}
J_{s}(u)= \La_{n,p}(1-\al) \int_{Q}\frac{|u(x)-u(y)|^p}{|x-y|^{n+p\al}} dxdy + \int_{\Om}|u|^p dx - \frac{s}{\e} \int_{\Om_{\e}} (u^{+})^p dx.
\end{align*}
Then $J_{s}\in C^{1}(\mc{W}^{\al,p},\mb R)$ and for any $\phi\in \mc{W}^{\al,p}$

\begin{equation*}
\langle J_{s}^{\prime}(u),\phi \rangle  =  p\text{ }\La_{n,p}(1-\al) \mathcal{H}_{\al,p}(u,\phi) 
 + p \int_{\Om}|u|^{p-2}u \phi dx - \frac{ps}{\e}
\int_{\Om_{\e}} (u^{+})^{ p-1}\phi dx.
\end{equation*}
\noi Also $\tilde{J_{s}}:= J_{s}|_{\mc S}$ is $C^1(\mc{W}^{\al,p},\mb R)$, where $\mc S$ is defined as
\[\mc S :=\left\{u\in \mc{W}^{\al,p}:\;    \; I(u):= \frac{1}{\e}\int_{\Om_{\e}}|u|^p=1\right\}.\]
\noi We first note that $u\in\mc S$ is a critical point of $\tilde{J_{s}}$ if and only if there exists $t\in \mb R$ such that

\begin{equation}\label{feq2}
 \La_{n,p}(1-\al) \mathcal{H}_{\al,p}(u,v) 
-\frac{s}{\e}\int_{\Om_{\e}} (u^{+})^{p-1} v dx = \frac{t}{\e}\int_{\Om_{\e}} |u|^{p-2} u v dx,
\end{equation}
for all $v\in \mc{W}^{\al,p}$. Hence $u\in \mc S$ is a nontrivial weak solution of the problem
\begin{equation*}
\La_{n,p}(1-\al)(-\De)_{p}^{\al} + |u|^{p-2}u = \frac{\chi_{\Om_{\e}}}{\e}\left((s+t)(u^{+})^{p-1} - {t}  (u^{-})^{p-1}\right)\; \text{in}\;
\Om, \; \mc{N}_{\al,p} u = 0 \;\mbox{in}\; \mb R^n \setminus \overline{\Om},
\end{equation*}
\noi which exactly means $(s+t,t)\in \sum_{p}$. Substituting $v=u$ in \eqref{feq2},
we get $t= \tilde{J_{s}}(u)$. Thus we obtain the following Lemma which links the critical point of $\tilde{J_{s}}$ and the spectrum $\sum_{p}$.
\begin{Lemma}
For $s\geq 0$, $(s+t,t)\in \mb R^2$ belongs to the spectrum $\sum_{p}$ if and only if there exists a critical point $u\in \mc S$ of
$\tilde{J_{s}}$ such that $t= \tilde{J_{s}}(u)$, a critical value.
\end{Lemma}
\begin{Proposition}
The first eigenfunction $\phi_{1,\e}$ is a global minimum for
$\tilde{J_{s}}$ with $\tilde{J_{s}}(\phi_{1,\e})=\la_{1,\e}-s$. The corresponding point in $\sum_{p}$ is $(\la_{1,\e},\la_{1,\e}-s)$ which lies on the
vertical line through $(\la_{1,\e},\la_{1,\e})$.
\end{Proposition}

\begin{proof}
We have
\begin{align*}
\tilde{J_{s}}(u)=& \; \La_{n,p}(1-\al) \int_{Q}\frac{|u(x)-u(y)|^p}{|x-y|^{n+p\al}} dxdy + \int_{\Om}|u|^{p}dx-\frac{s}{\e} \int_{\Om_{\e}} (u^{+})^p dx\\
\geq & \frac{\la_{1,\e}}{\e}  \int_{\Om_{\e}} |u|^p dx -\frac{s}{\e} \int_{\Om_{\e}} (u^{+})^p dx
\geq \la_{1,\e} -s.
\end{align*}
Thus $\tilde{J_{s}}$ is bounded below by $\la_{1,\e}-s$.
Moreover,
\[\tilde{J_{s}}(\phi_{1,\e})= \la_{1,\e} - \frac{s}{\e} \int_{\Om_{\e}} (\phi_{1,\e}^{+})^p dx = \la_{1,\e}-s.\]
Thus $\phi_{1,\e}$ is a global minimum of $\tilde{J_{s}}$ with
$\tilde{J_{s}}(\phi_{1,\e})=\la_{1,\e}-s$.\QED
\end{proof}
\begin{Proposition}\label{prop3}
The negative eigenfunction $-\phi_{1,\e}$ is a strict local minimum for $\tilde{J_{s}}$\\ with $\tilde{J_{s}}(-\phi_{1,\e})=\la_{1,\e}$. The
corresponding point in $\sum_{p}$ is $(\la_{1,\e}+s,\la_{1,\e})$, which lies on the horizontal line through $(\la_{1,\e},\la_{1,\e})$.
\end{Proposition}

\begin{proof}
Suppose by contradiction that there exists a sequence $u_k\in\mc  S$, $u_k\ne -\phi_{1,\e}$ with \\$\tilde{J_{s}}(u_k) \leq \la_{1,\e}$,
$u_k\ra -\phi_{1,\e}$ in $\mc{W}^{\al,p}$. We claim that $u_k$ changes sign for sufficiently large $k$. Since $u_k\ra -\phi_{1,\e}$, $u_k$ must be $<0$
for sufficiently large $k$. If $u_k \leq 0$ for a.e $x\in \Om$, then
\begin{align*}
\tilde{J_{s}}(u_k)= \La_{n,p}(1-\al) \int_{Q}\frac{|u_k(x)-u_k(y)|^p}{|x-y|^{n+p\al}} dxdy+\int_{\Om}|u_k|^{p} dx
>\la_{1,\e},
\end{align*}
since $u_k \not\equiv \pm\phi_{1,\e}$ and we get contradiction as
$\tilde{J_{s}}(u_k) \leq \la_{1,\e}$. Therefore the claim is proved. \\
Now, define
$w_k := \frac{\e^{\frac{1}{p}} u_{k}^{+}}{\|u_{k}^{+}\|_{L^p(\Om_{\e})}}$
and
\[r_k := \La_{n,p}(1-\al)\int_{Q}\frac{|w_{k}(x)-w_{k}(y)|^p}{|x-y|^{n+p\al}} dxdy + \int_{\Om}|w_k|^{p} dx. \]
\noi We claim that $r_k \ra \infty$ as $k\ra\infty$. Assume by contradiction that $r_k$ is bounded. Then there exists a
subsequence (still denoted by $\{w_k\}$) of $\{w_k\}$ and $w\in \mc{W}^{\al,p}$ such that $w_k \rightharpoonup w$ weakly in $\mc{W}^{\al,p}$ and $w_k\ra w$ strongly in $L^{p}(\Om)$. It implies $w_k\ra w$ strongly in $L^{p}(\Om_{\e})$. Therefore $\frac{1}{\e} \int_{\Om_{\e}} w^p dx =1$, $w\geq 0$ a.e in $\Om_{\e}$ and so for some $\eta>0$, $\de =|\{x\in \Om_{\e}: w(x)\geq \eta\}|>0$. Since, $u_k\ra -\phi_{1,\e}$ in $\mc{W}^{\al,p}$ and hence in $L^{p}(\Om)$. Therefore, for each $\eta>0$, $|\{x\in \Om_{\e} : u_k(x)\geq \eta\}|\ra 0$ as $k\ra \infty$
and $|\{x\in \Om_{\e} : w_k(x)\geq \eta\}|\ra 0$ as $k\ra\infty$, which is a contradiction to $\eta>0$. Hence, $r_k \ra \infty$.
Clearly, one can have
\begin{align*}
|u_k(x)-&u_k(y)|^p=(|u_k(x)-u_k(y)|^2)^{\frac p2}=[((u_{k}^{+}(x)-u_{k}^{+}(y))-(u_{k}^{-}(x)-u_{k}^{-}(y)))^2]^{\frac p2}\notag\\
=&[(u_{k}^{+}(x)-u_{k}^{+}(y))^2+ (u_{k}^{-}(x)-u_{k}^{-}(y))^2 - 2(u_{k}^{+}(x)-u_{k}^{+}(y))(u_{k}^{-}(x)-u_{k}^{-}(y))]^{\frac p2}\notag\\
=& [(u_{k}^{+}(x)-u_{k}^{+}(y))^2+ (u_{k}^{-}(x)-u_{k}^{-}(y))^2 + 2u_{k}^{+}(x) u_{k}^{-}(y)+ 2 u_{k}^{-}(x) u_{k}^{+}(y)]^{\frac p2}\notag\\
\geq & |u_{k}^{+}(x)-u_{k}^{+}(y)|^p +|u_{k}^{-}(x)-u_{k}^{-}(y)|^p.
\end{align*}
Using the above inequality, we have
\begin{align}\label{fe12}
\tilde{J_{s}}(u_k)=& \La_{n,p}(1-\al) \int_{Q}\frac{|u_k(x)-u_k(y)|^p}{|x-y|^{n+p\al}} dxdy + \int_{\Om}|u_k|^{p} - \frac{s}{\e} \int_{\Om_{\e}} (u^{+}_k)^p dx\notag\\
\geq & \left[ \La_{n,p}(1-\al) \int_{Q}\frac{|u_{k}^{+}(x)-u_{k}^{+}(y)|^p}{|x-y|^{n+p\al}}dxdy + \int_{\Om}|u_k^+|^{p}\right]\notag \\
&\quad +\left[\La_{n,p}(1-\al) \int_{Q}\frac{|u_{k}^{-}(x)-u_{k}^{-}(y)|^p}{|x-y|^{n+p\al}}dxdy + \int_{\Om}|u_k^-|^{p} - \frac{s}{\e} \int_{\Om_{\e}} (u^{+}_{k})^p dx\right]\notag\\
\geq &  \frac{(r_k-s)}{\e} \int_{\Om_{\e}} (u_{k}^{+})^p dx + \frac{\la_{1,\e}}{\e}\int_{\Om_{\e}} (u_{k}^{-})^p dx.
\end{align}
On the other hand, since $u_k\in \mc S$, we obtain
\begin{equation}\label{fe13}
\tilde{J_{s}}(u_k)\leq \la_{1,\e} =   \frac{\la_{1,\e}}{\e}\int_{\Om_{}\e} (u_{k}^{+})^p dx+
\frac{\la_{1,\e}}{\e}\int_{\Om_{}\e} (u_{k}^{-})^p dx.
\end{equation}
\noi From (\ref{fe12}) and (\ref{fe13}), we have
\[\frac{(r_k -s-\la_{1,\e})}{\e} \int_{\Om_{\e}} (u_{k}^{+})^p dx\leq 0,\]
\noi and this implies
\[r_k- s\leq \la_{1,\e},\]
\noi which contradicts the fact that $r_k\ra +\infty$. Therefore, $-\phi_{1,\e}$ is the strict local minimum.\QED
\end{proof}
\begin{Proposition}\label{fpr1}
\cite{AR} Let $Y$ be a Banach space, $g,f \in C^{1}(Y,\mb R)$, $M=\{u\in Y
\;|\; g(u)=1\}$ and $u_0$, $u_1\in M$. let $\e>0$ such that
$\|u_1-u_0\|>\e$ and \[\inf\{f(u): u\in M \;\mbox{and}\;
\|u-u_0\|_{Y}=\e\}>\max\{f(u_0),f(u_1)\}.\]
\noi Assume that $f$ satisfies the $(P.S)$ condition on $M$ and that
\[\Gamma =\{\ga \in C([-1,1], M): \ga(-1)=u_0 \;\mbox{and}\; \ga(1)=u_1\}\]
is non empty. Then $\ds c=\inf_{\ga \in \Gamma}\max_{u\in\ga[-1,1]}
f(u)$ is a critical value of $f|_M$.
\end{Proposition}
\noi We will now find the third critical point via mountain pass Theorem as stated above.
A norm of derivative of the restriction $\tilde{J_{s}}$ of $J_{s}$
at $u\in \mc S$ is defined as
\[\|\tilde{J}_{s}^{\prime}(u)\|_{*}=\min\{\|\tilde{J}_{s}^{\prime}(u)- t I^{\prime}(u)\|:  t\in \mb R\}.\]
\begin{Lemma}\label{fle21}
$J_{s}$ satisfies the $(P.S)$ condition on $\mc S$.
\end{Lemma}

\begin{proof}
Let $J_{s}(u_k)$ and $t_k\in \mb R$ be a sequences such that  for some $K>0$,
\begin{align}\label{feq3}
|J_{s}(u_k)|\leq K
\end{align}
and 
 \small{
\begin{align}\label{feq4}
\left|\Ga_{n,p}(1-\al) \mathcal{H}_{\alpha,p}(u_k, v)+\int_{\Om}|u_k|^{p-2}u_k v 
 - \frac{s}{\e}\int_{\Om_{\e}} (u^{+}_{k})^{p} v dx - \frac{t_k}{\e} \int_{\Om_{\e}}
|u_{k}|^{p-2} u_{k}v dx \right|\leq \eta_{k}\|v\|
\end{align}}
\noi for all $v\in \mc{W}^{\al,p}$, $\eta_k\ra 0$. From \eqref{feq3}, using fractional Sobolev embedding, we obtain $\{u_k\}$ is bounded in $\mc{W}^{\al,p}$ which implies there is a subsequence denoted by $u_k$ and $u_0\in \mc{W}^{\al,p}$ such that
$u_k\rightharpoonup u_0$ weakly in $\mc{W}^{\al,p}$, and $u_{k}\ra u_0$
strongly in $L^{p}(\Om)$ for all $1\leq p< p_{\al}^*$. Substituting $v=u_k$
in \eqref{feq4}, we obtain
\[|t_k|\leq \La_{n,p}(1-\al)\int_{Q}\frac{|u_k(x)-u_k(y)|^p}{|x-y|^{n+p\al}}dxdy + \int_{\Om}|u_k|^{p} +
\frac{s}{\e}\int_{\Om_{\e}} (u^{+}_{k})^p dx + \eta_{k}\|u_k\|\leq C.\]
\noi Hence, $t_k$ is bounded sequence so has a convergent subsequence
say $t_k$ that converges to $t$. \noi Next, we claim that $u_k\ra u_0$
strongly in $\mc{W}^{\al,p}$. Since $u_k\rightharpoonup u_0$ weakly in $\mc{W}^{\al,p}$, we
get
{\small\begin{align}\label{feq7}
\int_{Q}\frac{|u_0(x)-u_0(y)|^{p-2} (u_0(x)-u_0(y))(u_k(x)-u_k(y))}{|x-y|^{n+p\al}}dxdy&\lra \int_{Q} \frac{|u_0(x)-u_0(y)|^{p}}{|x-y|^{n+p\al}}dxdy
\end{align}}
as $k\rightarrow \infty$. Also $\ld \tilde{J_{s}^{\prime}}(u_k),(u_k-u_0)\rd= o(\eta_k)$. This is equivalently,
{\small\begin{align*}
&\left|\La_{n,p}(1-\al) \int_{Q}\frac{|u_k(x)-u_k(y)|^{p-2}(u_k(x)-u_k(y))((u_k- u_0)(x)-(u_k-u_0)(y))}{|x-y|^{n+p\al}}dxdy \right|\notag\\
& \leq o(\eta_k)+\|u_{k}\|_{L^p(\Om)}^{p-1}\|u_k-u_0\|_{L^p(\Om)}+ s\|u_{k}^{+}\|_{L^p(\Om_{\e})}^{p-1} \|u_k-u_0\|_{L^p(\Om_{\e})}+|t_k|\|u_k\|_{L^p(\Om_{\e})}^{p-1}\|u_k-u_0\|_{L^p(\Om_{\e})}\\
&\quad  \lra 0\;\mbox{as}\; k\ra\infty.
\end{align*}}
Thus,
{\small\begin{align}\label{feq8}
\int_{Q}\frac{|u_k(x)-u_k(y)|^p}{|x-y|^{n+p\al}}dxdy- \int_{Q}\frac{|u_k(x)-u_k(y)|^{p-2}(u_k(x)-u_k(y))(u_0(x)-u_0(y))}{|x-y|^{n+p\al}}dxdy \ra 0,
\end{align}}
\noi as $k\rightarrow \infty.$  As we know that $|a-b|^p \leq 2^p(|a|^{p-2}a-|b|^{p-2}b)(a-b)$  for all $a, b \in \mathbb{R}$.
Therefore, from \eqref{feq7} and \eqref{feq8} we obtain
\[\int_{Q}\frac{|(u_k-u_0)(x)-(u_k-u_0)(y)|^p}{|x-y|^{n+p\al}} dxdy\ra 0 \;\text{ as }\; k \ra \infty \]
Hence,  $u_k$ converges strongly to  $u_0$ in $\mc{W}^{\al,p}$.\QED
\end{proof}

\begin{Lemma}\label{fle22}
Let $\eta_0>0$  be such that
\begin{align}\label{feq13}
\tilde{J_{s}}(u)>\tilde{J_{s}}(-\phi_{1,\e}) \end{align} for all $u\in
B(-\phi_{1,\e}, \eta_0)\cap \mc S$ with $u \not\equiv -\phi_{1,\e}$, where the ball is
taken in $\mc{W}^{\al,p}$. Then for any $0<\eta<\eta_0$,
\begin{align}\label{feq14}
\inf\{\tilde{J_{s}}(u): u\in \mc S\;\mbox{and}\;
\|u-(-\phi_{1,\e})\|=\eta\}> \tilde{J_{s}}(-\phi_{1,\e}).
\end{align}
\end{Lemma}
\begin{proof}
If possible, let  infimum in \eqref{feq14} is equal to $\tilde{J_{s}}(-\phi_{1,\e})=\la_{1,\e}$ for some $\eta$ with $0<\eta<\eta_0$. It implies
there exists a sequence $u_k\in \mc S$ with $\|u_k-(-\phi_{1,\e})\|=\eta$ such that
\begin{align}\label{feq15}
\tilde{J_{s}}(u_k)\leq \la_{1,\e} + \frac{1}{2k^2}.
\end{align}
Consider the set $V =\{u\in\mc S: \eta-\de\leq\|u-(-\phi_{1,\e})\|\leq \eta+\de\}$, where $\de$ is choosen such
that $\eta-\de>0$ and $\eta+\de<\eta_0$. From \eqref{feq14} and given hypotheses, it follows that $\inf\{\tilde{J_{s}}(u): u\in V\}=\la_{1,\e}.$ Now for each $k$, we apply Ekeland's variational principle to the functional $\tilde{J_{s}}$ on $V$ to get the
existence of $v_k\in V$ such that
\begin{align}
\tilde{J_{s}}(v_k)&\leq \tilde{J_{s}}(u_k),\; \|v_k-u_k\| \leq \frac{1}{k}.\label{feq018}\\
\tilde{J_{s}}(v_k)&\leq \tilde{J_{s}}(u) +
\frac{1}{k}\|u-v_k\|,\;{\text{for all}}\; u\in V. \label{feq020}
\end{align}
\noi We claim that $v_k$ is a Palais-Smale sequence for $\tilde{J_{s}}$ on $\mc S$. That is, there exists $M>0$ such that $|\tilde{J_{s}}(v_k)|<M$  and
$\|\tilde{J}_{s}^{\prime}(v_k)\|_{*}\ra 0 \text{ as } k \ra \infty$. Once this is proved then by Lemma \ref{fle21}, there exists a subsequence denoted by $v_k$ of $v_k$ such that $v_k\ra v$ strongly in $\mc{W}^{\al,p}$. Clearly, $v\in \mc S$ and satisfies $\|v-(-\phi_{1,\e})\|\leq\eta+\de<\eta_0$ and $\tilde{J_{s}}(v)= \la_{1,\e}$ which contradicts \eqref{feq13}.\\

\noi Now, the boundedness of  $\tilde{J_{s}}(v_k)$  follows from $\eqref{feq15}$ and \eqref{feq018}. So, we only need to prove that
$\|\tilde{J}_{s}^{\prime}(v_k)\|_{*}\ra 0$. Let $k>\frac{1}{\de}$ and take $w\in \mc{W}^{\al,p}$ tangent to $\mc S$ at $v_k$. That is, $\ds \frac{1}{\e} \int_{\Om_{\e}} |v_k|^{p-2} v_k w dx =0.$ Then by taking  
$\displaystyle u_t:= \frac{\e^{\frac{1}{p}}(v_k+tw)}{\|v_k+tw\|_{L^p(\Om_{\e})}}$ for $t\in \mb R$, we get  
\begin{align*}
\lim_{t\ra 0}\|u_t-(-\phi_{1,\e})\|= \|v_k-(-\phi_{1,\e})\|\leq \|v_k-u_k\|+\|u_k-(-\phi_{1,\e})\|\leq \frac{1}{k}+\eta<\de+\eta,
\end{align*}
and
\begin{align*}
\lim_{t\ra 0}\|u_t-(-\phi_{1,\e})\|= \|v_k-(-\phi_{1,\e})\|\geq\|u_k-(-\phi_{1,\e})\|-\|v_k -u_k\|\geq \eta-\frac{1}{k}>\eta-\de.
\end{align*}
\noi Hence, for $t$ small enough $u_t\in  V$ and replacing $u$ by $u_t$ in \eqref{feq020}, we get
\begin{align*}
\tilde{J_{s}}(v_k)&\leq \tilde{J_{s}}(u_t)+\frac{1}{k}\|u_t-v_k\|.
\end{align*}
Let $r(t): = \e^{\frac 1p}\|v_k + tw\|_{L^p(\Om_{\e})}$, then 
\begin{align*}
\frac{J_{s}(v_k)-J_{s}(v_k+tw)}{t}&\leq \frac{J_{s}(u_t)+\frac{1}{k}\|u_t-v_k\|-J_{s}(v_k+tw)}{t}\\
&=\frac{1}{k\; t\; r(t)}\|v_k(1-r(t)+tw)\|+\frac{1}{t}\left(\frac{1}{r(t)^p}-1\right)J(v_k+tw).
\end{align*}
Now since
$\frac{d}{dt}r(t)^p\ds|_{t=0}= \frac{p}{\e} \int_{\Om_{\e}} |v_k|^{p-2} v_k w=0,$
we obtain  $\frac{r(t)^p-1}{t}\ra 0\;\mbox{as}\; t\ra 0,$ and then  $\frac{1-r(t)}{t}\ra 0\; \mbox{as}\; t\ra 0.$ Therefore, we get
\begin{align}\label{feq16}
|\ld{J_{s}^{\prime}}(v_k),w\rd|\leq \frac{1}{k}\|w\|.
\end{align}
\noi Since $w$ is arbitrary in $\mc {W}^{\al,p}$, we choose $a_k$ such that $\frac{1}{\e} \int_{\Om_{\e}} |v_k|^{p-2} v_k(w-a_k v_k) dx =0$. Replacing $w$ by $w-a_k v_k$ in \eqref{feq16}, we get $\left|\ld{J_{s}^{\prime}}(v_k),w\rd - a_k \ld{J_{s}^{\prime}}(v_k), v_k\rd\right|\leq \frac{1}{k}\|w-a_k v_k\|$. Since $\|a_k v_k\|\leq C\|w\|$, we obtain $\left|\ld{J_{s}^{\prime}}(v_k),w \rd - t_k \int_{\Om}|v_k|^{p-2}v_k w dx \right|\leq \frac{C}{k}\|w\|$, where $t_k=\ld{J_{s}^{\prime}}(v_k), v_k \rd$. Hence, $\|\tilde{J_{s}^{\prime}}(v_k)\|_*\ra 0\;\mbox{as}\; k\ra\infty,$ as we required. \QED
\end{proof}

\begin{Proposition}\label{fga}
Let $\mc{W}^{\al,p}$ be a Banach Space. Let $\eta>0$ such that
$\|\phi_{1,\e}-(-\phi_{1,\e})\|>\eta$ and
\[\inf\{\tilde{J_s}(u): u\in \mc S \;\mbox{and}\;
\|u-(-\phi_{1,\e})\|=\eta\}>\max\{\tilde{J_s}(-\phi_{1,\e}),
\tilde{J_s}(\phi_{1,\e})\}.\] Then
$\Gamma =\{\ga \in C([-1,1], \mc S): \ga(-1)=- \phi_{1,\e} \;\mbox{and}\;
\ga(1)=\phi_{1,\e}\}$
is non empty and
\begin{align}\label{feq18}
c(s)=\inf_{\ga \in \Gamma}\max_{u\in\ga[-1,1]} J_{s}(u)
\end{align}
is a critical value of $\tilde{J_{s}}$. Moreover $c(s)>\la_{1,\e}$.
\end{Proposition}

\begin{proof}
We prove that $\Ga$ is non-empty. To end this, we take $\phi\in \mc{W}^{\al,p}$
such that $\phi\not\in \mb R\phi_{1,\e}$ and consider the path
$t\phi_{1,\e}+(1-|t|)\phi$ then  $\displaystyle w=  \frac{\e^{\frac{1}{p}} (t\phi_{1,\e}+(1-|t|)\phi)}{\| t\phi_{1,\e}+(1-|t|)\phi\|_{L^p(\Om_{\e})}}$. Moreover the (P.S)
condition and the geometric assumption are satisfied by the Lemmas
\ref{fle21} and \ref{fle22}. Then by Proposition \ref{fpr1}, $c(s)$ is
a critical value of $\tilde{J_s}$. Using the definition of $c(s)$ we
have $c(s)>\max\{\tilde{J_s}(-\phi_{1,\e}),
\tilde{J_s}(\phi_{1,\e})\}=\la_{1,\e}$.\QED
\end{proof}
\noi Thus we have proved the following:
\begin{Theorem}\label{fth1}
For each $s\geq 0$, the point $(s+c(s),c(s))$, where $c(s)>\la_{1,\e}$ is
defined by the minimax formula \eqref{feq18}, then the point
$(s+c(s),c(s))$ belongs to $\sum_{p}$.
\end{Theorem}

\noi It is a trivial fact that $\sum_{p}$ is symmetric with respect to
diagonal. The whole curve, that we obtain using Theorem \ref{fth1}
is denoted by
\[\mc C:= \{(s+c(s),c(s)),(c(s),s+c(s)): s\geq 0\}.\]

\section{First Nontrivial Curve}
We start this section by establishing that the lines $\mb R\times\{\la_{1,\e}\}$ and $\{\la_{1,\e}\}\times \mb R$ are isolated in $\sum_{p}$. Then
we state some topological properties of the functional $\tilde{J_{s}}$ and some Lemmas. Finally, we prove that the curve $\mc C$
constructed in the previous section is the first non trivial curve in the spectrum $\sum_{p}$. As a consequence of this, we also obtain a variational characterization of the second eigenvalue $\la_{2,\e}$.
\begin{Proposition}
The lines $\mb R\times \{\la_{1,\e}\}$ and $\{\la_{1,\e}\}\times \mb R$ are isolated in $\sum_{p}$. In other words, there exists no sequence $(a_k,
b_k)\in \sum_{p}$ with $a_k>\la_{1,\e}$ and $b_k > \la_{1,\e}$ such that $(a_k, b_k)\ra (a,b)$ with $a=\la_{1,\e}$ or $b=\la_{1,\e}$.
\end{Proposition}
\begin{proof}
Suppose by contradiction that there exists a sequence $(a_k, b_k)\in\sum_{p}$ with \\$a_k,$ $b_k>\la_{1,\e}$ and $(a_k, b_k)\ra (a, b)$ with $a$ or
$b=\la_{1,\e}$. Let $u_k\in \mc {W}^{\al,p}$ be a solution of
{\small\begin{equation}\label{feq17}
\La_{n,p}(1-\al)(-\De)_{p}^{\al} u_k  + |u_k|^{p-2}u_k = \frac{\chi_{\Om_\e}}{\e} (a_k(u_k^{+})^{p-1} - b_k(u_k^{-})^{p-1})\;\text{in}\;
\Om,\; \mc{N}_{\al,p} u_k =0 \;\mbox{in}\; \mb R^n \setminus \overline{\Om},
\end{equation}}
 with $\frac{1}{\e}\int_{\Om_{\e}} |u_k|^pdx=1$. Multiplying by $u_k$ in \eqref{feq17} and integrate, we have
 \[\La_{n,p}(1-\al) \int_{Q}\frac{|u_{k}(x)-u_{k}(y)|^p }{ |x-y|^{n+p\al}}dxdy + \int_{\Om}|u_k|^{p}dx = \frac{a_k}{\e} \int_{\Om_{\e}}(u_k^+)^p dx - \frac{b_k}{\e}\int_{\Om_{\e}}(u_k^-)^p dx\leq a_k.\]
Thus $\{u_k\}$ is a bounded sequence in $\mc {W}^{\al,p}$. Therefore upto
a subsequence $u_k \rightharpoonup u$ weakly in $\mc {W}^{\al,p}$ and $u_k\ra u$
strongly in $L^p(\Om_{\e})$. Then taking limit $k\ra\infty$ in the weak formultion of  \eqref{feq17}, we obtain
{\small\begin{equation}\label{feq017}
\La_{n,p}(1-\al)(-\De)_{p}^{\al} u  + |u|^{p-2}u = \frac{\chi_{\Om_\e}}{\e} (\la_{1,\e} (u^{+})^{p-1} - b (u^{-})^{p-1}) \; \text{ in }\;
\Om,\; \mc{N}_{\al,p} u = 0 \;\mbox{in}\; \mb R^n \setminus \overline{\Om}.
\end{equation}}
\noi Taking $u^+$ as test function  in \eqref{feq017}  we obtain
{\small\begin{equation}\label{f02}
\La_{n,p}(1-\al) \mathcal{H}_{\al,p}(u,u^+) 
 +\int_{\Om} (u^{+})^{p}dx = \frac{\la_{1,\e}}{\e}\int_{\Om_{\e}}(u^+)^p dx.
\end{equation}}
Observe that
\begin{align}\label{03}
((u(x)-u(y))(u^{+}(x)-u^{+}(y))= 2 u^{-}(x)u^{+}(y)+(u^{+}(x)-u^{+}(y))^2
\end{align}
and
\begin{align}\label{033}
|u(x)-&u(y)|^{p-2}= (|u(x)-u(y)|^2)^{\frac{p-2}2}\notag\\
&=(|u^{+}(x)-u^{+}(y)|^2+ |u^{-}(x)-u^{-}(y)|^2 +2
u^{+}(x)u^{-}(y)+ 2u^{+}(y)u^{-}(x))^{\frac{p-2}2}\notag\\
&\geq |u^{+}(x)-u^{+}(y)|^{p-2}.
\end{align}
\noi Using \eqref{03} and \eqref{033} in \eqref{f02} and using the definition of $\la_{1,\e}$, we obtain
\begin{align*}
\frac{\la_{1,\e}}{\e}\int_{\Om_{\e}} (u^+)^p dx\leq \La_{n,p}(1-\al) \int_{Q}\frac{|u^{+}(x)-u^{+}(y)|^p}{ |x-y|^{n+p\al}}dxdy +\int_{\Om} (u^{+})^{p}dx \leq \frac{\la_{1,\e}}{\e}\int_{\Om_{\e}} (u^+)^p dx.
\end{align*}
Thus,
\[\La_{n,p}(1-\al) \int_{Q}\frac{|u^{+}(x)-u^{+}(y)|^p}{ |x-y|^{n+p\al}} dxdy +\int_{\Om} (u^{+})^{p}dx =\frac{\la_{1,\e}}{\e}\int_{\Om_{\e}} (u^+)^p dx,\] so either $u^+\equiv 0$ or $u=\phi_{1,\e}$. If $u^+\equiv 0$ then $u\leq 0$ and $\eqref{feq017}$ implies that $u$ is an eigenfunction with $u\leq 0$ so that $u=-\phi_{1,\e}$. So, in any case $u_k$ converges to either $\phi_{1,\e}$ or $-\phi_{1,\e}$ in $L^{p}(\Om_{\e})$. Thus
\begin{align}\label{feq019}
\mbox{either} \;|\{x\in \Om_{\e} : u_k(x)<0\}|\ra 0\; \mbox{ or}\;|\{x\in\Om_{\e} :u_k(x)> 0\}|\ra 0 \quad \text{as } k\ra \infty .
\end{align}
\noi On the other hand, taking $u_k^+$ as test function in \eqref{feq17}, we obtain

\begin{align}\label{feq016}
\La_{n,p}(1-\al) \mathcal{H}_{\al,p}(u_k,u_{k}^{+}) 
+\int_{\Om} |u_{k}|^{p-2}u_{k}u_{k}^+ 
 = \frac{a_k}{\e}\int_{\Om_{\e}}(u_{k}^+)^p.
\end{align}
\noi Using H\"{o}lders inequality, fractional Sobolev embeddings and \eqref{feq016}, we obtain
{\small\begin{align*}
\La_{n,p}&(1-\al) \int_{Q}\frac{|u^{+}_k(x)-u^{+}_k(y)|^p}{|x-y|^{n+p\al}}dxdy +\int_{\Om} (u_{k}^{+})^{p}dx\\
&\leq\La_{n,p}(1-\al) \int_{Q}\frac{|u_{k}(x)-u_{k}(y)|^{p-2}(u_{k}(x)-u_{k}(y))(u^{+}_k(x)-u^{+}_k(y))}{|x-y|^{n+p\al}}dxdy +\int_{\Om} |u_{k}|^{p-2}u_{k}u_{k}^+dx \\
&=\La_{n,p}(1-\al) \mathcal{H}_{\al,p}(u_k, u_{k}^{+})+\int_\Om |u_k|^{p-2} u_k u_{k}^{+}dx\\
&= \frac{a_k}{\e}\int_{\Om_{\e}}(u_{k}^+)^p dx\\
&\leq \frac{a_k}{\e} C |\{x\in \Om_{\e} : u_k(x)>0\}|^{1-\frac{p}{q}}\|u_k^+\|^{p}
\end{align*}}
 with a constant $C>0$, $p<q\leq p^*=\frac{np}{n-p\al}$. Then we have
\[ |\{x\in \Om : u_k(x)>0\}|^{1-\frac{p}{q}} \geq \e a_k ^{-1}C^{-1}\min\{\La_{n,p}(1-\al),1\}.\]
Similarly, one can show that
\[|\{x\in \Om : u_k(x)<0\}|^{1-\frac{p}{q}}\geq \e b_k ^{-1}C^{-1}\min\{\La_{n,p}(1-\al),1\}.\]
Since $(a_k,b_k)$ does not belong to the trivial lines $\la_{1,\e}\times\mb R$ and $\mb
R\times \la_{1,\e}$ of $\sum_{p}$,  by  \eqref{feq17} we conclude that $u_k$ changes sign. Hence, from the above inequalities, we get a
contradiction with \eqref{feq019}. Therefore, the trivial lines $\la_{1,\e}\times \mb R$ and $\mb R\times \la_{1,\e}$ are isolated in $\sum_{p}$.\QED
\begin{Lemma}\label{fle1}\cite{cfg}
Let $\mc S= \{u\in \mc {W}^{\al,p} : \frac{1}{\e}\int_{\Om_{\e}}|u|^p dx =1\}$ then
\begin{enumerate}
\item $\mc S$ is locally arcwise connected.
\item Any open connected subset $\mc O$ of $\mc S$ is arcwise connected.
\item If $\mc O^{'}$ is any connected component of an open set $\mc O\subset \mc S$, then $\partial \mc O^{\prime}\cap \mc O=\emptyset$.
\end{enumerate}
\end{Lemma}
\begin{Lemma}\label{fle01}
Let $\mc O=\{u\in \mc S : \tilde{J_s}(u)<r\}$, then any connected component of $\mc O$ contains a critical point of $\tilde{J_s}$.
\end{Lemma}
\begin{proof}
Let $\mc O_1$ be any connected component of $\mc O$, let $d=\inf\{\tilde{J_{s}}(u): u\in  \overline{\mc O}_1\}$, where
$\overline{\mc O}_1$ denotes the closure of $\mc O_1$ in $\mc W^{\al,p}$. We show that there exists $u_0\in \mc W^{\al,p}$ such that
$\tilde{J_{s}}(u_0)=d$. For this let $u_k\in \mc O_1$ be a minimizing sequence such that $\tilde{J_{s}}(u_k)\leq d+\frac{1}{2k^2}$. For each $k$, by applying Ekeland's Variational principle, we get a sequence $v_k\in \overline{\mc O}_1$ such that
\begin{align*}
\tilde{J_{s}}(v_k)&\leq \tilde{J_{s}}(u_k),\; \|v_k-u_k\| \leq \frac{1}{k},\;\mbox{and}\;\tilde{J_{s}}(v_k)\leq \tilde{J_{s}}(v) + \frac{1}{k}\|v-v_k\|\;\text{for all}\;  v\in \overline{\mc O}_1.
\end{align*}
For $k$ large enough, we have
\[\tilde{J_{s}}(v_k)\leq \tilde{J_{s}}(u_k)\leq d+ \frac{1}{2k^2}<r,\]
then $v_k\in \mc O$. By Lemma \ref{fle1}, we get $v_k\not\in \partial\mc O_1$ so $v_k\in \mc O_1$. On the other hand, for $t$ small enough and $w$ such that
$\frac{1}{\e}\int_{\Om_{\e}}  |v_k|^{p-2}v_k w dx=0$, we have
$\displaystyle u_t:=\frac{\e^{\frac{1}{p}}(v_k+tw)}{\|v_k+tw\|_{L^p{(\Om_{\e})}}}\in \overline{\mc O}_1.$ Then $\tilde{J_{s}}(v_k)\leq \tilde{J_{s}}(u_t) + \frac{1}{k}\|u_t-v_k\|.$
 Following the same calculation as in Lemma \ref{fle22}, we have that $v_k$ is a Palais-Smale sequence for
$\tilde{J_{s}}$ on $\mc S$ i.e $\tilde{J_{s}}(v_k)$ is bounded and $\|\tilde{J_{s}}(v_k)\|_{*}\ra 0$. Again by Lemma \ref{fle21}, up to
a subsequence $v_k\ra u_0$ strongly in $\mc W^{\al,p}$ and hence $\tilde{J_{s}}(u_0)=d<r$ and moreover $u_0\in \mc O$. By part $3$ of
Lemma \ref{fle1}, $u_0\not\in \partial \mc O_1$ so $u_0\in \mc O_1$. Hence $u_0$ is a critical point of $\tilde{J_{s}}$, which completes
the proof.\QED
\end{proof}
Before proving the main Theorem \ref{f31}, we state some Lemmas and the details of the proof can be found in \cite{second} and \cite{pal}.
\begin{Lemma}\label{f2}\cite[Lemma B.1]{second}
Let $1 \leq p \leq \infty$ and $U,V \in \mb R$ such that $U.V \leq 0$. Define the following function
\[g(t)=|U -tV|^p+|U-V|^{p-2}(U-V)V|t|^p,\; t \in \mb R.\]
Then we have \[g(t)\leq g(1)=|U-V|^{p-2}(U-V)U,\; t \in \mb R.\]
\end{Lemma}
\begin{Lemma}\label{f3}\cite[Lemma 4.1]{pal}
Let $\al\in (0,1)$ and $p>1$. For any non-negative functions $u$, $v \in \mc W^{\al,p}$, consider the function $\sigma_t:= \left[(1-t)v^p(x)+ tu^p(x)\right]^{\frac{1}{p}}$ for all $t\in[0, 1]$. Then
\begin{equation*}
[\sigma_t]_{\al,p} \leq (1-t)[v]_{\al,p} + t[u]_{\al,p},\;\mbox{for all}\; t \in [0,1], \;\text{where}\; [u]_{\al,p}\; \text{is defined in}\; \eqref{feq12}.
\end{equation*}
\end{Lemma}

\noi {\bf Proof of Theorem \ref{f31}:} Assume by contradiction that there exists $\mu$ such that $\la_{1,\e}<\mu<c(s)$ and $(s+\mu,\mu)\in \sum_{p}$. Using the fact that $\{\la_{1,\e}\}\times \mb R$ and $\mb R\times \{\la_{1,\e}\}$ are isolated in $\sum_{p}$ and $\sum_{p}$ is closed we can choose such a point with $\mu$ minimum. In other words, $\tilde{J_{s}}$ has a critical value $\mu$ with $\la_{1,\e}<\mu<c(s)$, but there is no critical value in
$(\la_{1,\e},\mu)$. If we construct a path connecting from $\phi_{1,\e}$ to $-\phi_{1,\e}$ such that $\tilde{J_s}\leq \mu$, then we get a
contradiction with the definition of $c(s)$, which will complete the proof.

\noi Let $u\in \mc S$ be a critical point of $\tilde{J_s}$ at level
$\mu$. Then $u$ satisfies,
\begin{align}\label{feq20}
&\La_{n,p}(1-\al) \int_{Q} \frac{|u(x)-u(y)|^{p-2}(u(x)-u(y))(v(x)-v(y))}{ |x-y|^{n+p\al}} dxdy+\int_{\Om}|u|^{p-2}uvdx \notag\\
&\qquad = \frac{(s+\mu)}{\e}\int_{\Om_{\e}}(u^{+})^{p-1} vdx -\frac{\mu}{\e} \int_{\Om_{\e}}(u^{-})^{p-1}vdx
\end{align}
\noi for all $v\in \mc {W}^{\al,p}$. Substituting $v=u^{+}$ in \eqref{feq20}, we have
{\small\begin{equation}\label{feq21}
\La_{n,p}(1-\al) \int_{Q} \frac{|u(x)-u(y)|^{p-2}(u(x)-u(y))(u^+(x)-u^+(y))}{ |x-y|^{n+p\al}} dxdy+\int_{\Om}(u^+)^{p}dx = \frac{(s+\mu)}{\e}\int_{\Om_{\e}}(u^{+})^p dx.
\end{equation}}
Since, $|u^+(x)-u^{+}(y)|^p \leq|u(x)-u(y)|^{p-2}(u(x)-u(y))(u^+(x)-u^+(y)$, we obtain
\[ \La_{n,p}(1-\al)\int_{Q} \frac{|u^+(x)-u^{+}(y)|^p}{ |x-y|^{n+p\al}} dxdy+\int_{\Om}(u^+)^pdx-\frac{s}{\e}\int_{\Om_{\e}}(u^{+})^pdx \leq \mu.\]
\noi Again substituting  $v=u^-$ in \eqref{feq20}, we have
{\small\begin{equation}\label{feq22}
\La_{n,p}(1-\al) \int_{Q} \frac{|u(x)-u(y)|^{p-2}(u(x)-u(y))(u^-(x)-u^-(y))}{ |x-y|^{n+p\al}} dxdy-\int_{\Om}(u^-)^{p}dx = -\frac{\mu}{\e}\int_{\Om_{\e}} (u^{-})^p dx.
\end{equation}}
Therefore,
{\small\begin{equation*}
\La_{n,p}(1-\al) \int_{Q} \frac{|u(x)-u(y)|^{p-2}((u^-(x)-u^-(y))^2+2u^+(x)u^-(y))}{ |x-y|^{n+p\al}} dxdy+\int_{\Om}(u^-)^{p}dx = \frac{\mu}{\e}\int_{\Om_{\e}}(u^{-})^p dx
\end{equation*}}
Since, we know that $|u^-(x)-u^-(y)|^{p}\leq |u(x)-u(y)|^{p-2}\left[(u^-(x)-u^-(y))^2+2u^+(x)u^-(y)\right].$
It implies
\[ \La_{n,p}(1-\al)\int_{Q} \frac{|u^-(x)-u^{-}(y)|^p}{ |x-y|^{n+p\al}} dxdy+\int_{\Om}|u^-|^pdx\leq \mu.\]
 Therefore, from all above relations, one can easily verify that
\[\tilde{J_{s}}(u)=\mu,\; \tilde{J_{s}}\left(\frac{\e^\frac{1}{p} u^+}{ \|u^+\|_{L^p(\Om_{\e})}}\right)
\leq \mu,\tilde{J_{s}}\left(\frac{\e^\frac{1}{p}u^-}{\|u^-\|_{L^p(\Om_{\e})}}\right)\leq \mu -s,\tilde{J_{s}}\left( \frac{-\e^\frac{1}{p}u^-}{\|u^-\|_{L^p(\Om_{\e})}}\right)\leq \mu .\]
Since, $u$ changes sign (see Proposition \ref{prop3}), the following paths are well-defined on $\mc S$:
\[u_{1}(t)=\frac{u^+- (1-t)u^-}{\e^\frac{-1}{p} \|u^+- (1-t)u^-\|_{L^p(\Om_{\e})}},\quad u_2(t)=\frac{[(1-t)(u^{+})^p+ t(u^{-})^p]^{1/p}}{\e^\frac{-1}{p} \|(1-t)(u^{+})^p+ t(u^{-})^p\|_{L^p(\Om_{\e})}}\]
\[u_3(t)=\frac{(1-t)u^{+}-u^{-}}{\e^\frac{-1}{p} \|(1-t)u^{+}-u^{-}\|_{L^p(\Om_{\e})}}.\]
Then, using the above calculations and Lemma \ref{f2} for $U=u^+(x)-u^+(y)$ and\\ $V=u^-(x)-u^-(y)$, one can easily obtain that for all $t\in[0,1]$,
\begin{align*}
\tilde{J_{s}}(u_1(t)) &\leq \frac{\La_{n,p}(1-\al)\displaystyle\int_{Q}\frac{|U-V|^{p-2}(U-V)U}{|x-y|^{n+p\al}}+\int_{\Om}(u^+)^p-\frac{s}{\e}\int_{\Om_{\e}}(u^+)^p }{\e^{-1} \|u^+- (1-t)u^-\|_{L^p(\Om_{\e})}^p}\\
&\quad +\frac{|1-t|^p\left[-\La_{n,p}(1-\al)\displaystyle\int_{Q}\frac{|U-V|^{p-2}(U-V)V}{|x-y|^{n+p\al}}+\int_{\Om}(u^-)^p\right]}{\e^{-1}\|u^+- (1-t)u^-\|_{L^p(\Om_{\e})}^p}\\
&=\mu,
\end{align*}
\noi by using \eqref{feq21} and \eqref{feq22}.
Now using Lemma \ref{f3} we have,
\begin{equation*}
\begin{split}
\tilde{J_{s}}(u_2(t)) &\leq \frac{(1-t)\left[\La_{n,p}(1-\al)\displaystyle\int_{Q}\frac{|u^+(x)-u^+(y)|^p}{|x-y|^{n+p\al}}+\int_{\Om}(u^+)^p -\frac{s}{\e}\int_{\Om_{\e}}(u^+)^p\right] }{\e^{-1} \|(1-t)(u^{+})^p+ t(u^{-})^p\|^p_{L^p(\Om_{\e})}}\\
& \quad +\frac{t \left[\La_{n,p}(1-\al)\displaystyle\int_{Q}\frac{|u^-(x)-u^-(y)|^p}{|x-y|^{n+p\al}}+\int_{\Om}(u^-)^p- \frac{s}{\e} \int_{\Om_{\e}}(u^-)^p\right]}{\e^{-1} \|(1-t)(u^{+})^p+ t(u^{-})^p\|^p_{L^p(\Om_{\e})}}\\
& \leq \mu - \frac{s t \int_{\Om_{\e}}(u^-)^p}{\e^{-1} \|(1-t)(u^{+})^p+ t(u^{-})^p\|^p_{L^p(\Om_{\e})}}\\& \leq \mu.
\end{split}
\end{equation*}
Again, by Lemma \ref{f2}, for $U=u^-(y)-u^-(x)$ and $V=u^+(y)-u^+(x)$, we obtain
{\begin{align*}
\tilde{J_{s}}(u_3(t)) &\leq \frac{\La_{n,p}(1-\al)\displaystyle\int_{Q}\frac{|U-V|^{p-2}(U-V)U}{|x-y|^{n+p\al}}+\int_{\Om}(u^-)^p  }{\e^{-1} \|(1-t)u^{+}-u^{-}\|_{L^p(\Om_{\e})}^p}\\&\quad + \frac{ |1-t|^p \left[ -\La_{n,p}(1-\al) \displaystyle \int_{Q}\frac{|U-V|^{p-2}(U-V)V}{|x-y|^{n+p\al}}+ \int_{\Om}(u^+)^p-\frac{s}{\e}\int_{\Om_{\e}}(u^+)^p \right]}{\e^{-1}\|(1-t)u^{+}-u^{-}\|_{L^p(\Om_{\e})}^p}\\
&=\mu,\; \mbox{by using}\; \eqref{feq21}\;\mbox{and}\; \eqref{feq22}.
\end{align*}}
\noi Let $\mc O = \{v\in\mc S: \tilde{J_s}(v)<\mu-s\}$. Then clearly
$\phi_{1,\e}\in \mc O$, while $-\phi_{1,\e}\in \mc O$ if $\mu- s>\la_{1,\e}$.
Moreover $\phi_{1,\e}$ and $-\phi_{1,\e}$ are the only possible critical
points of $\tilde{J_s}$ in $\mc O$ because of the choice of $\mu$. We note that
$\tilde{J_s}\left(\frac{\e^{\frac{1}{p}} u^-}{\|u^-\|_{L^p(\Om_{\e})}}\right)\leq \mu-s$,
$\frac{\e^{\frac{1}{p}} u^-}{\|u^-\|_{L^p(\Om_{\e})}} $ does not change sign and vanishes on a
set of positive measure, it is not a critical point of
$\tilde{J_s}$. Therefore, there exists a $C^1$ path $\eta:[-\de,\de]\ra
\mc S$ with $\eta(0)= \frac{\e^{\frac{1}{p}}u^-}{\|u^-\|_{L^p(\Om_{\e})}}$ and
$\frac{d}{dt}\tilde{J_s}(\eta(t))|_{t=0}\ne 0$. Using this path we
can move from $\frac{\e^{\frac{1}{p}} u^-}{\|u^-\|_{L^p(\Om_{\e})}}$ to a point $v$ with
$\tilde{J_s}(v)<\mu-s$. Taking a connected component of $\mc O$
containing $v$ and applying Lemma \ref{fle01} we have that either
$\phi_{1,\e}$ or $-\phi_{1,\e}$ is in this component. Let us assume that it is
$\phi_{1,\e}$. So we continue by a path $u_{4}(t)$ from
$\frac{\e^{\frac{1}{p}} u^-}{\|u^-\|_{L^p(\Om_{\e})}}$ to $\phi_{1,\e}$ which is at level
less than $\mu$. Then the path $-u_{4}(t)$ connects
$\frac{-\e^{\frac{1}{p}} u^-}{\|u^-\|_{L^p(\Om_{\e})}}$ to $-\phi_{1,\e}$. We observe
that $|\tilde{J_s}(u)- \tilde{J_s}(-u)|\leq s.$ Then it follows that
\[\tilde{J_s}(-u_4(t))\leq \tilde{J_s}(u_4(t))+s\leq \mu-s+s= \mu \;\mbox{for all}\; t.\]
Connecting $u_1(t)$, $u_2(t)$ and $u_4(t)$, we get a path from $u$
to $\phi_{1,\e}$ and joining $u_3(t)$ and $-u_4(t)$ we get a path from
$u$ to $-\phi_{1,\e}$. These yields a path $\ga(t)$ on $\mc S$ joining
from $\phi_{1,\e}$ to $-\phi_{1,\e}$ such that $\tilde{J_s}(\ga(t))\leq \mu$
for all $t$, which concludes the proof.\QED
\end{proof}
As a consequence of Theorem 1.1, we give a variational characterization of the second value of \eqref{eq2}.
\begin{Corollary}
The second eigenvalue $\la_2$ of \eqref{eq2} has the variational characterization given by
\begin{align*}
  \la_2 := & \inf_{\ga \in \Gamma}\sup_{u\in \ga}\left(\La_{n,p}(1-\al)\int_{Q}\frac{|u(x)-u(y)|^p}{|x-y|^{n+p\al}} dxdy +\int_{\Om}|u|^p dx\right),
\end{align*}
where $\Ga$ is as in Proposition \ref{fga}.
\end{Corollary}
\begin{proof}
Taking $s=0$ in Theorem \ref{f31}, we have $c(0)=\la_2$ and \eqref{feq18} concludes the proof.\QED
\end{proof}

\section{Properties of the curve $\mathcal{C}$}
In this section, we prove that the curve $\mc C$ is Lipschitz
continuous, has a certain asymptotic behavior and is strictly
decreasing.

\noi For $A \subset \Om_{\e}$, define the eigenvalue problem
\begin{equation}\label{feq51}
 \La_{n,p}(1-\al)(-\De)^{\al}_p u +|u|^{p-2}u = \frac{\chi_{A}}{\e}(\la |u|^{p-2}u) \; \text{in}\;
\Om, \quad \mc{N}_{\al,2} u = 0 \; \quad  \mbox{in}\; \mb R^n \setminus \overline{\Om},
\end{equation}
Let $\la_{1,\e}(A)$ denotes the first eigenvalue of \eqref{feq51}, then
\[\la_{1,\e}(A) = \inf_{u\in \mc W^{\al,p}}\left\{
\La_{n,p}(1-\al)\int_{Q}\frac{|u(x)-u(y)|^p}{|x-y|^{n+p\al}}dxdy + \int_{\Om} |u|^p dx : \int_{A}|u|^p dx= \e \right\}.\]
\begin{Lemma}\label{f111}
Let $A$, $B$ be two bounded open sets in $\Om_{\e}$, with
$A\subsetneq B$ and $B$ is connected then $\la_{1,\e}(A)>\la_{1,\e}(B)$.
\end{Lemma}
\begin{proof}
Clearly from the definition of $\la_{1,\e}$, we have $\la_{1,\e}(A)\geq\la_{1,\e}(B)$. Let if possible equality
holds and let $\phi_{1,\e}$ be a non-negative normalized eigenfunction
associated to $\la_{1,\e}(A)$ such that $\phi_{1,\e}$ is equal to zero outside
$A$. Therefore, from the definition of $\la_{1,\e}(A)$, we have
\[\La_{n,p}(1-\al)\int_{Q} \frac{|\phi_{1,\e}(x)-\phi_{1,\e}(y)|^p}{|x-y|^{n+p\al}}dxdy+\int_{\Om}|\phi_{1,\e}|^pdx = \frac{\la_1(A)}{\e}\int_{A} \phi_{1,\e}^p dx = \frac{\la_1(B)}{\e}\int_{B} \phi_{1,\e}^p dx.\]
It implies $\phi_{1,\e}$ is an eigenfunction associated to $\la_{1,\e}(B)$. But
this is impossible since $B$ is connected and $\phi_{1,\e}$ vanishes on
$B\setminus A \ne\emptyset$.\QED
\end{proof}
\begin{Proposition}
The curve $s\ra (s+c(s), c(s))$, $s\in \mb R^+$ is Lipschitz continuous and strictly decreasing $($in the sense that $s_1<s_2$ implies $s_1+c(s_1)<s_2+c(s_2)$ and	$c(s_1)>c(s_2))$.
\end{Proposition}

\begin{proof}
Let $s_1<s_2$ then $\tilde{J}_{s_1}(u)>\tilde{J}_{s_2}(u)$ for all $u\in \mc S$. So we have $c(s_1)>c(s_2)$. Now for every $\eta>0$ there exists $\ga\in\Gamma$ such that $\ds\max_{u\in \ga[-1,1]}\tilde{J}_{s_2}(u)\leq c(s_2)+\eta,$
and so
\[0\leq c(s_1)- c(s_2)\leq \max_{u\in \ga[-1,1]}\tilde{J}_{s_1}(u)- \max_{u\in \ga[-1,1]}\tilde{J}_{s_2}(u)+\eta.\]
Let $u_0\in\ga[-1,1]$ such that $\ds\max_{u\in \ga[-1,1]}\tilde{J}_{s_1}(u)=\tilde{J}_{s_1}(u_0)$. Then
\[0\leq c(s_1)- c(s_2)\leq \tilde{J}_{s_1}(u_0)-\tilde{J}_{s_2}(u_0)+\eta \leq s_2-s_1+\eta,\]
as $\eta>0$ is arbitrary so the curve $\mc C$ is Lipschitz continuous
with constant $\leq 1$. \\
\noi Next, in order to prove that the curve is decreasing, it suffices to argue for $s>0$. Let $0<s_1<s_2$ then it implies $c(s_1)> c(s_2)$. On the other hand, since $(s_1+c(s_1), c(s_1))$, $(s_2+c(s_2), c(s_2))\in\sum_p$, Theorem \ref{f31} implies that $s_1+c(s_1)< s_2+c(s_2)$, which completes the proof.\QED
\end{proof}
\noi As $c(s)$ is decreasing and positive so the limit of $c(s)$ exists as $s\ra \infty$.
\begin{Theorem}
If $n\geq p \al$, then the limit of $c(s)$ as $s\ra \infty$ is	$\la_{1,\e}$.
\end{Theorem}
\begin{proof}
For $n\geq  p \al $, we can choose a function $\phi\in \mc W^{\al,p} $ such	that there does not exist $r\in\mb R$ such that $\phi(x)\leq r \phi_{1,\e}(x)$ a.e. in $\Om_{\e}$(it suffices to take $\phi\in \mc W^{\al,p}$ such that it is unbounded from above in a neighborhood of some point $0\ne x\in\Om_{\e}$). Suppose that the result is not true then there exists $\de>0$ such that
$\ds\max_{u\in \ga[-1,1]}\tilde{J}_{s}(u)\geq \la_{1,\e} +\de \;\mbox{for all}\;\ga\in \Ga\;\mbox{and all}\; s\geq 0.$ Consider a path	$\ga\in \Ga$ by
\[\ga(t)= \frac{\e^{\frac{1}{p}}(t\phi_{1,\e} +(1-|t|)\phi)}{\|t\phi_{1,\e} +(1-|t|)\phi\|_{L^{p}(\Om_{\e})}}\;\text{for all}\; t\in[-1,1].\]
Now, for every $s>0$, let $t_s\in [-1,1]$ satisfy
$\ds\max_{t\in[-1,1]}\tilde{J_s}(\ga(t))= \tilde{J_s}(\ga(t_s))$. Let $v_{t_s}= {t_s\phi_{1,\e} +(1-|t_s|)\phi}$. Then we have
\begin{align}\label{feq4.1}
\tilde{J_s}(v_{t_s})\geq \frac{(\la_{1,\e}+\de)}{\e}\int_{\Om_{\e}} |v_{t_s}|^p.
\end{align}
Letting $s\ra \infty$, we can assume a subsequence $t_s\ra \tilde t\in[-1,1]$. Then $v_{t_s}$ is bounded in $\mc W^{\al,p}$. So, from last inequality
we obtain $\int_{\Om_{\e}}(v_{t_s}^{+})^p dx \ra 0$ as $s\ra\infty$, which forces
\[\int_{\Om_{\e}}((\tilde{t}\phi_{1,\e} +(1-|\tilde{t}|)\phi)^+)^p dx=0.\]
Hence, $\tilde{t}\phi_{1,\e} +(1-|\tilde{t}|)\phi \leq 0$. By the choice of $\phi$, $\tilde{t}$ must be equal to $-1$. Passing to the limit in \eqref{feq4.1}, we obtain
\begin{align*}
\frac{\la_{1,\e}}{\e}\int_{\Om_{\e}}\phi_{1,\e}^{p}dx =\La_{n,p}(1-\al) \int_{Q}\frac{|\phi_{1,\e}(x)-\phi_{1,\e}(y)|^p}{|x-y|^{n+p\al}} dxdy+\int_{\Om}|\phi_{1,\e}|^p dx  \geq \frac{(\la_{1,\e}+\de)}{\e}\int_{\Om_{\e}} |v_{t_s}|^p.
\end{align*}
We arrive at a contradiction that $\delta \leq 0$. Hence $c(s)\ra \la_{1,\e}$ as $s\ra\infty$.\QED
\end{proof}
\section{Non Resonance between $(\la_1,\la_1)$ and $\mc C$}

\noi In this section, we will study the non-resonance problem with respect to the Fu\v{c}ik spectrum for $p=2$ case.
\begin{Lemma}\label{fle31}
Let $(a,b)\in \mc C$, and let $m(x)$, $b(x)\in L^{\infty}(\Om)$ satisfying
\begin{align}\label{feq31}
\la_{1,\e}\leq m(x)\leq a,\;\;\la_{1,\e}\leq b(x)\leq b.
\end{align}
 Assume that
\begin{align}\label{feq32}
\la_{1,\e}<m(x)\;\mbox{and}\; \la_{1,\e}<b(x)\;\mbox{ on subsets of positive measure of } \Om_{\e}.
\end{align}
Then any non-trivial solution $u$ of
\begin{equation}\label{feq33}
\La_{n,2}(1-\al)(-\De)^{\al} u +u = \frac{\chi_{\Om_{\e}}}{\e}(m(x)u^{+} - b(x) u^{-}) \; \text{in}\;
\Om, \quad \mc{N}_{\al,2} u = 0 \;  \mbox{ in }\; \mb R^n \setminus \overline{\Om},
\end{equation}
changes sign in $\Om_{\e}$ and \[m(x)=a\; \mbox{a.e. on}\;\{x\in \Om_{\e} :
u(x)>0\}\;\quad\mbox{and}\quad \;b(x)=b\; \mbox{a.e. on}\;\{x\in \Om_{\e} : u(x)<0\}.\]
\end{Lemma}

\begin{proof}
Let $u$ be a nontrivial solution of $\eqref{feq33}$. Replacing $u$ by	$-u$ if necessary. we can assume that the point $(a,b)$ in $\mc C$ is such that $a\geq b$. We first claim that $u$ changes sign in $\Om_{\e}$. Suppose by contradiction that this is not true, first consider the case $u\geq 0$, (case $u\leq 0$ can be proved similarly). Then $u$ solves
\[\La_{n,2}(1-\al)(-\De)^{\al}u +u = \frac{\chi_{\Om_{\e}}}{\e}m(x) u^{+} \; \text{ in }\;\Om, \quad \mc{N}_{\al,2} u = 0 \;  \mbox{ in }\; \mb R^n \setminus \overline{\Om}.\]
This means that $u $ is an eigenfunction of the problem with weight $m(x)$ corresponding to the eigenvalue equal to one. From the definition of the first eigenvalue of the problem with weight $m(x)\geq \la_{1,\e}$, we have
\begin{equation}\label{feq34} 
\la_{1,\e}(m(x))=  \inf_{0 \not \equiv u\in \mc W^{\al,2}  } \left\{ \frac{\La_{n,2}(1-\al)\int_{Q} \frac{|u(x)-u(y)|^2}{|x-y|^{n+2\al}} dxdy +\int_{\Om}|u|^2(x) dx}{\frac{1}{\e}\int_{\Om_{\e}}m(x)  |u|^2 dx}:  \right\}=1.
\end{equation}
\noi From $\eqref{fle31}$, $\eqref{feq31}$ and $\eqref{feq34}$, we have	
\begin{align*}
1=&\frac{\La_{n,2}(1-\al)\int_{Q} {|\phi_{1,\e}(x)-\phi_{1,\e}(y)|^2}{|x-y|^{-(n+2\al)}} dxdy+\int_{\Om}|\phi_{1,\e}|^2(x) dx}{\la_{1,\e}}\\
>&\frac{\La_{n,2}(1-\al)\int_{Q} {|\phi_{1,\e}(x)-\phi_{1,\e}(y)|^2}{|x-y|^{-(n+2\al)}} dxdy +\int_{\Om}|\phi_{1,\e}|^2(x) dx}{ \frac{1}{\e}\int_{\Om_{\e}} m(x)|\phi_{1,\e}|^2 dx}\geq 1,
\end{align*}
which is a contradiction. Hence, $u$ changes sign on $\Om_{\e}$.\\
\noi Let suppose by contradiction that either
\begin{align}\label{feq35}
|\{x\in \Om_{\e} : m(x)<a\; \mbox{and}\; u(x)>0\}|>0
\end{align}
or
\begin{align}\label{feq36}
|\{x\in \Om_{\e} : b(x)<b\; \mbox{and}\; u(x)<0\}|>0.
\end{align}
\noi Suppose that $\eqref{feq35}$ holds (a similar argument will hold for $\eqref{feq36})$. Put $a-b= s\geq 0$. Then $b= c(s)$, where $c(s)$
is given by $\eqref{feq18}$. We show that there exists a path $\ga\in\Ga$ such that
\begin{align}\label{feq37}
\max_{u\in \ga[-1,1]}\tilde{J}_{s}(u)<b,
\end{align}
which gives a contradiction with the definition of $c(s)$, prove the last part of the Lemma.

\noi In order to construct $\ga$ we show that there exists of a function $v\in \mc W^{\al,2}$ such that it changes sign and satisfies
\begin{align}\label{feq38}
&\frac{\La_{n,2}(1-\al)\int_{Q} {|v^{+}(x)-v^{+}(y)|^2}{|x-y|^{-(n+2\al)}} dxdy + \int_{\Om}(v^{+})^2 dx}{\frac{1}{\e} \int_{\Om_{\e}} (v^{+})^2 dx}
<a \quad \mbox{and}\notag\\
&\quad\frac{\La_{n,2}(1-\al)\int_{Q} {|v^{-}(x)-v^{-}(y)|^2}{|x-y|^{-(n+2\al)}}
dxdy +\int_{\Om}(v^{-})^2 dx}{\frac{1}{\e}\int_{\Om_{\e}} (v^{-})^2 dx} <b.
\end{align}
Let  $\mc O_1$ be a component of $\{x\in \Om_{\e} : u(x)>0\}$ such that  $|\{x\in\mc O_1:m(x)<a\}| >0$ and $\mc O_2$ be a component of $\{x\in \Om_{\e} : u(x)<0\}$ such that $|\{x\in\mc O_2:b(x)<b\}| >0$ . Define the eigenvalue problem
\begin{equation}\label{feq42}
\La_{n,2}(1-\al)(-\De)^{\al} u +u = \frac{\chi_{\mc O_{i}}}{\e}(\la u) \; \text{ in }\;
\Om, \quad \mc{N}_{\al,2} u = 0 \; \quad  \mbox{ in }\; \mb R^n \setminus \overline{\Om},\quad i=1,2.
\end{equation}
Let $\la_{1,\e}(\mc O_i)$ denotes the first eigenvalue of  \eqref{feq42}. Next, we claim that
\begin{align}\label{feq39}
\la_{1,\e}(\mc O_1)<a \quad \mbox{and} \quad \la_{1,\e}(\mc O_2)< b,
\end{align}
\noi where $\la_{1,\e}(\mc O_i)$ denotes the first eigenvalue of $\La_{n,2}(1-\al)(-\De)^{\al} u+u$  on $\mc W^{\al,2}$ and
\begin{align*}
\la_{1,\e}(\mc O_1)=&\frac{\La_{n,2}(1-\al)\int_{Q} {|u(x)-u(y)|^2}{|x-y|^{-(n+2\al)}} dxdy+ \int_{\Om}|u|^2 dx}{\frac{1}{\e}\int_{\mc O_1} |u|^2 dx}\\
&<a\frac{\La_{n,2}(1-\al)\int_{Q} {|u(x)-u(y)|^2}{|x-y|^{-(n+2\al)}} dxdy+ \int_{\Om}|u|^2 dx}{\frac{1}{\e}\int_{\mc O_1} m(x)  |u|^2 dx} =a,
\end{align*}
since $|x\in \mc O_1: m(x)< a|>0$. This implies $\la_{1,\e}(\mc O_1)<a$. The other inequality can be proved similarly. Now with some modification on the sets $\mc O_1$ and $\mc O_2$, we construct the sets $\tilde{\mc O}_1$ and $\tilde{\mc O}_2$ such that $\tilde{\mc O}_1\cap \tilde{\mc O}_2=\emptyset$ and $\la_{1,\e}(\tilde{\mc O}_1)<a$ and $\la_{1,\e}(\tilde{\mc O}_2)<b$. For $\nu\geq 0$, ${\mc O}_1(\nu)=\{x\in {\mc O}_1 : dist(x, (\Om_{\e})^c )>\nu\}.$ By Lemma $\ref{f111}$, we have $\la_{1,\e}(\mc O_1(\nu))\geq \la_{1,\e}(\mc O_1))$ and moreover $\la_{1,\e}(\mc O_1(\nu))\ra \la_{1,\e}(\mc O_1))$ as $\nu\ra0$. Then there exists $\nu_0>0$ such that
\begin{equation}\label{feq40}
\la_{1,\e}(\mc O_1(\nu))<a\;\mbox{ for all }\;0\leq \nu\leq \nu_0.
\end{equation}
Let $x_0\in \partial \mc O_2\cap \Om_{\e}$ (is not empty as $\mc O_1\cap\mc O_2=\emptyset)$ and choose $0<\nu<\min\{\nu_0, dist(x_0,\Om_{\e}^c)\}$ and $\tilde{\mc O}_1=\mc O_1(\nu)$ and $\tilde{\mc O}_2=\mc O_2 \cup B(x_0, \frac{\nu}{2})$. Then $\tilde{\mc O}_1\cap \tilde{\mc O}_2=\emptyset$ and by $\eqref{feq40}$, $\la_{1,\e}(\tilde{\mc O}_1)<a$. Since $\tilde{\mc O}_2$ is connected then by $\eqref{feq39}$ and Lemma \ref{f111}, we get $\la_1(\tilde{\mc O}_2)<b$. Now, we define $v=v_1-v_2$, where $v_i$ are the eigenfunctions associated to $\la_{i,\e}(\tilde{\mc O}_i)$. Then $v$ satisfies $\eqref{feq38}$.

\noi Thus there exist $v\in \mc W^{\al,2}$ which changes sign, satisfies condition \eqref{feq38}. Moreover we have
\begin{align*} \displaystyle
\tilde{J_s}\left(\frac{\e^{\frac{1}{2}} v}{\|v\|_{L^2(\Om_{\e})}}\right)=& \displaystyle \frac{\La_{n,2}(1-\al)\int_{Q}\frac{|v^{+}(x)-v^{+}(y)|^2}{|x-y|^{n+2\al}} dxdy+ \int_{\Om}(v^{+})^2}{\frac{1}{\e}\|v\|^2_{L^2(\Om_{\e})}}-s\frac{\int_{\Om_{\e}}(v^{+})^2 dx}{\|v\|^2_{L^2(\Om_{\e})}}\\
&\; +\frac{\La_{n,2}(1-\al)\int_{Q}\frac{|v^{-}(x)-v^{-}(y)|^2}{|x-y|^{n+2\al}} dxdy+ \int_{\Om}(v^{-})^2 dx}{\frac{1}{\e}\|v\|^2_{L^2(\Om_{\e})}}\\ &\;+4\frac{\La_{n,2}(1-\al)\int_{Q} \frac{v^{+}(x)v^{-}(y)}{|x-y|^{n+2\al}} dxdy}{\frac{1}{\e}\|v\|^2_{L^2(\Om_{\e})}}\\ <& (a-s)\frac{\int_{\Om_{\e}}(v^{+})^2 dx}{\|v\|^2_{L^2(\Om_{\e})}}+b\frac{\int_{\Om_{\e}}(v^{-})^2 dx}{\|v\|^2_{L^2(\Om_{\e})}}=b.
\end{align*}
\[\tilde{J_s}\left(\frac{\e^{\frac{1}{2}} v^+}{\|v^+\|_{L^2(\Om_{\e})}}\right)< a-s =b,\quad\tilde{J_s}\left(\frac{\e^{\frac{1}{2}} v^-}{\|v^-\|_{L^2(\Om_{\e})}}\right)< b-s. \]
Using Lemma \ref{fle01}, we have that there exists a critical point in the connected component of the set $\mc O=\{u\in\mc S:\tilde{J_{s}}(u)<b-s\}$. As the point $(a,b)\in \mc C$, the only possible critical point is $\phi_{1,\e}$, then we can construct a path from $\phi_{1,\e}$ to $-\phi_{1,\e}$ exactly in the same manner as in Theorem $\ref{f31}$ only by taking $v$ in place of $u$. Thus we have construct
a path satisfying $\eqref{feq37}$, and hence the result follows.\QED
\end{proof}
\begin{Corollary}\label{fle32}
Let $(a,b)\in \mc C$ and let $m(x)$, $b(x)\in L^{\infty}(\Om)$
satisfying $\la_{1,\e}\leq m(x)\leq a$ a.e, $\la_{1,\e}\leq b(x)\leq b$ a.e.
Assume that $\la_{1,\e}<m(x)$ and $\la_{1,\e}<b(x)$ on subsets of positive
measure on $\Om_{\e}$. If either $m(x)<a$ a.e in $\Om_{\e}$ or $b(x)<b$ a.e. in $\Om_{\e}$.
Then \eqref{feq33} has only the trivial solution.
\end{Corollary}
\begin{proof}
By the Lemma $\ref{fle31}$, any non-trival solution of $\eqref{feq33}$ changes sign and $m(x)=a$ a.e on $\{x\in\Om_{\e}:u(x)>0\}$ or $b(x)=b$ a.e on
$\{x\in \Om_{\e}: u(x)<0\}$. So, by our hypotheses, $\eqref{feq33}$ has only trivial solution.\QED
\end{proof}

\noi Now, we study the non-resonance between $(\la_1,\la_1)$ and $\mc C$
\begin{align}\label{041}
\left\{
\begin{array}{lr}
\La_{n,2}(1-\al)(-\De)^{\al} u+u=  \frac{\chi_{\Om_{\e}}f(x,u)}{\e} \;\mbox{ in }\;\Om,\quad \mc N_{\al,2}u=0\mbox{ in }\; \mb R^n \setminus \overline{\Om},
\end{array}
\right.
\end{align}
where $f(x,u)/u$ lies asymptotically between $(\la_{1,\e},\la_{1,\e})$ and $(a,b)\in \mc C$. Let $f:\Om\times\mb R\ra \mb R$ be a function satisfying $L^{\infty}(\Om)$ Caratheodory conditions.  Given a point $(a,b)\in \mc C$, we assume following:
\begin{align}\label{41}
\ga_{\pm}(x)\leq \liminf_{s\ra \pm\infty}\frac{f(x,s)}{s}\leq \limsup_{s\ra\pm \infty} \frac{f(x,s)}{s}\leq \Ga_{\pm}(x)
\end{align}
hold uniformly with respect to $x$, where $\ga_{\pm}(x)$ and $\Ga_{\pm}(x)$ are $L^{\infty}(\Om)$ functions which satisfy
\begin{align}\label{42}
\left\{\begin{array}{lr}
\quad \la_{1,\e}\leq \ga_{+}(x)\leq \Ga_{+}(x)\leq a\;\mbox{a.e. in}\;\Om_{\e}\\
\quad \la_{1,\e}\leq \ga_{-}(x)\leq \Ga_{-}(x)\leq b\;\mbox{a.e. in}\;\Om_{\e}.
\end{array}\right.
\end{align}
Write $F(x,s)=\int_{0}^{s}f(x,t) dt$, we also assume the following inequalities:
\begin{align}\label{43}
 \de_{\pm}(x)\leq \liminf_{s\ra \pm\infty}\frac{2F(x,s)}{|s|^2}\leq \limsup_{s\ra\pm \infty} \frac{2F(x,s)}{|s|^2}\leq \De_{\pm}(x)
 \end{align}
hold uniformly with respect to $x$, where $\de_{\pm}(x)$ and
$\De_{\pm}(x)$ are $L^{\infty}(\Om)$ functions which satisfy
\begin{equation}\label{44}
\left\{\begin{array}{lr}
\quad \la_{1,\e}\leq \de_{+}(x)\leq \De_{+}(x)\leq a\;\mbox{a.e. in}\;\Om_{\e},\quad \la_{1,\e}\leq \de_{-}(x)\leq \De_{-}(x)\leq b\;\mbox{a.e. in}\;\Om_{\e},\\
\quad \de_{+}(x)>\la_{1,\e} \;\mbox{and}\;\de_{-}(x)>\la_{1,\e} \;\mbox{on subsets of positive measure,}\\
\quad \mbox{either}\; \De_{+}(x)< a \;\mbox{a.e. in}\;\Om_{\e}\;\mbox{or}\;\De_{-}(x)< b \;\mbox{a.e. in}\;\Om_{\e}.
\end{array}\right.
\end{equation}
\begin{Theorem}\label{th51}
Let \eqref{41}, \eqref{42}, \eqref{43} and \eqref{44} hold and
$(a,b)\in\mc C$. Then problem \eqref{041} admits at least one
solution $u$ in $\mc W^{\al,2}$.
\end{Theorem}
Define the energy functional $\Psi:\mc W^{\al,2}\ra \mb R$ as
\[\Psi(u)=\frac{\La_{n,2}(1-\al)}{2}\int_{Q} \frac{|u(x)-u(y)|^2}{|x-y|^{n+2\al}}dxdy + \frac{1}{2}\int_{\Om}|u|^2 dx - \frac{1}{\e}\int_{\Om_{\e}} F(x,u) dx\]
Then $\Psi$ is a $C^{1}$ functional on $\mc W^{\al,2}$ and for all $v\in \mc W^{\al,2}$
\[\ld\Psi^{\prime}(u),v\rd=\La_{n,2}(1-\al) \int_{Q} \frac{(u(x)-u(y))(v(x)-v(y))}{|x-y|^{n+2\al}}dxdy +\int_{\Om}u v dx -\frac{1}{\e}\int_{\Om_{\e}}  {f(x,u)v} dx\]
and critical points of $\Psi$ are exactly the weak solutions of
\eqref{041}.\\
Now, we will state some Lemmas, for details of proof see Lemma 5.2 and 5.3 of \cite{sa}.
\begin{Lemma}\label{le51}
 $\Psi$ satisfies the $(P.S)$ condition on $\mc W^{\al,2}$.
 \end{Lemma}
\begin{Lemma}\label{le52}
There exists $R>0$ such that
\begin{align*}
\max\{\Psi(R\phi_{1,\e}),\Psi(-R\phi_{1,\e})\} < \max_{u\in \ga[-1,1]}\Psi(u),
\end{align*}
for any $\ga\in \Ga_1:=\{\ga\in C([-1,1],\mc S) : \ga(\pm 1)= \pm
R\phi_{1,\e}\}$.
\end{Lemma}

\noi {\bf Proof of Theorem \ref{th51}:} Lemmas \ref{le51} and
\ref{le52} complete the proof.\QED

 \end{document}